\documentclass[12pt]{amsart}
\usepackage{geometry}                
\usepackage{graphicx}
\usepackage{amssymb}

\usepackage{epstopdf}
\usepackage{amsmath}
\usepackage{amsthm}
\usepackage{amssymb}
\usepackage{indentfirst}
\theoremstyle{definition}
\newtheorem{thm}{Theorem}
\newtheorem{definition}{Definition}

\newtheorem{lem}{Lemma}

\newtheorem{rem}{Remark}

\title{On total flexibility of local structures of Finsler tori without conjugate points}
\author{Dong Chen}
\address{Dong Chen: Department of Mathematics, Pennsylvania State University, University Park, PA 16802, USA}
\email{chen\_d@math.psu.edu}
\date{}                                         

\begin{document}
\maketitle
\begin{abstract}
 We show that given a point on a Finsler surface, one can always find a
neighborhood of the point and isometrically embed this neighborhood
into a Finsler torus without conjugate points.
\end{abstract}

\section{Introduction}
 
 In this paper we study the universality of local structures of 2-dimensional Finsler tori without conjugate points. It is known that 2-dimensional Riemannian tori without conjugate points are flat, which was proved by E.Hopf ([8]) in 1940s. Hopf's paper is a partial answer to a question asked by Hedlund and Morse([7]), that is, whether the same result still holds in all dimensions. The positive answer to this question is now known as Hopf's conjecture. After that many other people studied this problem with various assumptions. In 1994, D.Burago and S.Ivanov([1]) proved the Hopf's conjecture. This breakthrough shows the rigidity of Riemannian tori without conjugate points. 

However, there are examples of non-flat Finsler tori without conjugate points, thus the original Hopf's conjecture does not hold in Finsler case. For example, one can construct such a non-flat Finsler 2-torus by making symplectic (contact) perturbations on the Euclidean torus or as surfaces of revolution.

Before Burago and Ivanov, C. Croke and B. Kleiner([4]) have shown that  in Riemannian case, smoothness (or bi-Lipchitz) of the Heber foliation([6]) would imply that a torus without conjugate points is flat. Smoothness of the Heber foliation is (more or less) equivalent to the assertion that the geodesic flow is smoothly conjugate to that of some flat Finsler torus. It is still an open question if the Heber foliation of a Finsler manifold without conjugate points  is smooth, and whether the geodesic flow of such manifold is smoothly conjugate to that of some flat Finsler torus. We even do not know whether every geodesic with an irrational rotation number is dense in a 2-dimensional Finsler torus without conjugate points.
\footnote{2010 \emph{Mathematics Subject Classification.} 53C23, 53C60.

\emph{Key words and phrases.} Finsler geometry, conjugate points, enveloping function.

The author was partially supported by Dmitri Burago's NSF grant DMS-0905838.}

In this note I extend an approach suggested by Burago-Ivanov to show that there are no local restrictions for  a metric to be the metric of a Finsler torus without conjugate points. Namely, I prove the following theorem:
\begin{thm}
Suppose $(M, \varphi)$ is a  $C^k(k\geq 3)$ Finsler surface. Then for any $ p_0\in M$, we can find a neighborhood $U$ of $p_0$, and an isometric embedding $\Psi: (U, \varphi|_{TU}) \rightarrow(\mathbb{T}^2, \tilde{\varphi})$, where $(\mathbb{T}^2,\tilde{\varphi})$ is a  $C^k$ Finsler torus without conjugate points. If in addition, $\varphi$ is symmetric, then $\tilde{\varphi}$ can be chosen to be symmetric too.
\end{thm}
 Let $\gamma$ be a ray with unit speed in a Finsler manifold. Define the Busemann function $b_{\gamma}: M\rightarrow \mathbb{R}$ with respect to $\gamma$ by
$$b_{\gamma}(x):=\lim_{t\rightarrow \infty}(t-d(x,\gamma(t))).$$
The main idea in the proof of the theorem is to generalize the concept of Busemann functions on Finsler manifold to an enveloping function. Such extension does not depend on the ray.  And we can get back the Finsler metric from the enveloping function. By perturbing the enveloping function we can get a perturbation of the Finsler metric.
\section {Preliminaries}
\begin{definition}
([5])\textbf{A \boldmath{$C^k(k\geq 2)$} Minkowski norm on  \boldmath{$\mathbb{R}^n$}} is a nonnegative real-valued function $\varphi$ on $\mathbb{R}^n$ with the following properties: 

(i) $\varphi$ is $C^k$ on $\mathbb{R}^n\backslash\{0\}$;

(ii) $\varphi(\lambda v)= \lambda \varphi(v), $ for $ \lambda>0$;

(iii) The $n\times n$ matrix
\begin{displaymath}
(g_{ij}):=(\frac{1}{2}\frac{\partial^2}{\partial y_i \partial y_j} \varphi^2)
\end{displaymath}
is positive definite at every point except the origin.

If moreover we have:

(iv) $\varphi(-v)=\varphi(v),$

then $\varphi$ is called a \textbf{symmetric  \boldmath{$C^k$} Minkowski norm}.

\end{definition}
\begin{rem}
By (i) and (ii), we have $\varphi(0)=0$. By (iii), we know $\varphi(y)=0$ implies $y=0$. The condition (iii) is equivalent to the fact that the unit sphere $S \subseteq\mathbb{R}^n$ is quadratically convex, i.e., every point on $S$ has positive principal curvatures. 

\end{rem}

\begin{definition}
Let $M$ be an $n$-dimensional $C^\infty$ manifold. For $k\geq 2$,  \textbf{A  \boldmath{$C^k$} Finsler structure on  \boldmath{$M$}} is a function $\varphi:TM\rightarrow [0,\infty)$ with the following properties:

(i) $\varphi$ is $C^k$ on the complement of the zero section;

(ii) $\varphi|_{T_xM}$ is a Minkowski norm for all $x\in M$.

A manifold with a $C^k$ Finsler structure is called \textbf{a  \boldmath{$C^k$} Finsler manifold} (or just a Finsler manifold). 

If moreover, for all $x\in M$, $\varphi|_{T_xM}$ is a symmetric Minkowski norm, then the Finsler structure is called \textbf{symmetric}.

\end{definition}

\begin{rem}
If $\varphi$ is a Minkowski norm on $\mathbb{R}^n$, then we have the triangle inequality: For all $v_1,v_2\in\mathbb{R}^n$, $\varphi(v_1+v_2)\leq \varphi(v_1)+\varphi(v_2)$. The equality holds iff $y_2=\alpha y_1$ or $y_1=\alpha y_2$ for some $\alpha\geq0$. Hence symmetric Minkowski norms are norms in the sense of functional analysis.
\end{rem}

If $\gamma:[a,b]\rightarrow M$ is a smooth curve on a Finsler manifold $(M, \varphi)$, then one defines the length of $\gamma$ by
\begin{displaymath}
L(\gamma):= \int^b_a \varphi(\gamma(t), \gamma'(t))dt.
\end{displaymath}
Using this definition of length we define a non-symmetric metric(i.e. a positive definite function on $M\times M$ satisfying the triangle inequality) on $M$ by letting the distance $d(x,y)$ from $x$ to $y$ be the infimum of the lengths of all piecewise smooth curves starting from $x$ and ending at $y$. It is non-symmetric  since $d(x, y)$ may not be equal to $d(y, x)$. Under this non-symmetric metric we can define geodesics in the following way: a curve $\gamma: [a, b]\rightarrow M$ is said to be a geodesic of $(M, \varphi)$ if for every sufficiently small interval $[c, d]\subseteq [a, b]$,  $\gamma|_{[c,d]}$ realizes the distance from $\gamma(c)$ to $\gamma(d)$. In this paper we will always assume that a geodesic is unit-speed, i.e. if $\gamma$ is a geodesic, then $\varphi(\gamma(s), \gamma'(s))=1$, for $s\in[a,b]$. A geodesic $\gamma: [a, b]\rightarrow M$ is called \textbf{minimal} if for $a\leq t_1<t_2\leq b, d(\gamma(t_1), \gamma(t_2))=t_2-t_1$. And a $C^k$ Finsler metric $\varphi$ on $M$ is called $\textbf{simple}$ if every pair of points on $M$ is connected by a unique geodesic depending $C^k$ smoothly on the endpoints.

\begin{definition}
A function $f$ on a Finsler manifold $(M, \varphi)$ is called \textbf{forward 1-Lip-chitz} if for $p, q\in M, f(p)-f(q)\leq d(q, p)$.
\end{definition}

\section{Enveloping function}

We use some notation and techniques from [3]. To make this note more reader-friendly, we copy them here.

Let $(M, \varphi)$ be a Finsler manifold. We have a norm $\varphi^*$ on the cotangent bundle $T^*M$ given by:
\begin{displaymath}
\varphi^*(\alpha):=\sup\{\alpha(v)|v\in T_xM, \varphi(v)=1\},
\end{displaymath}
for $x\in M, \alpha\in T^*_x M$. And we denote by $UM$ and $U^*M$ the bundles of unit spheres of $\varphi$ and $\varphi^*$. Since $\varphi$ is Minkowski on each tangent space, $\varphi^*$ is also Minkowski on each cotangent space, hence $U^*_xM$ is quadratically convex for all $x\in M$. A $C^1$ function on $M$ is called \textbf{distance-like} if $\varphi^*(d_x f)=1$ for all $x\in M$. 

Notice that a distance-like function is always forward 1-Lipchitz. In fact, for any $x,y\in M$ and any unit-speed curve $c:[a,b]\rightarrow M$ starting at $x$ and ending at $y$, if $f$ is distance-like, then 
$$f(y)-f(x)=\int_a^b df_{c(s)}(c'(s))ds\leq b-a=L(c)$$
By taking the infimum for all $c$, $f$ is  forward 1-Lipchitz.\\

Let $D$ denote the n-dimensional disc with a Finsler metric $\varphi$. If $\varphi$ is simple, then the distance function $d(p,\cdot)$ of a point $p\in\partial D$ is distance-like. 

\begin{definition}
A continuous function $F:\partial D\times D\rightarrow\mathbb{R}$ is called a  \textbf{ \boldmath{$C^k$} enveloping function} for $\varphi$ if $F$ is $C^k$ smooth outside the diagonal of $\partial D\times \partial D$ and the following conditions are satisfied:

(a) For every $p\in\partial D$, the function $F_p:=F(p,\cdot)$ is distance-like.

(b) For every $x\in D$, the map $p\rightarrow d_x F_p$ is a diffeomorphism from $\partial D$ to $U_x^*D$.
\end{definition}

If $\varphi$ is a $C^k$ simple Finsler metric and $F$ is given by $F(p,x):=d(p, x)$, then $F$ is a $C^k$ enveloping function. On the other hand, given an enveloping function $F$ we can define a distance function on $D\times D$ by
$$d_F(x,y):=\sup_{p\in\partial D}F(p,y)-F(p,x).$$
By $d_F$ we can define a metric $\varphi_F$ on $TD$, and the unit sphere of $\varphi_F^*$ in $T_x^*D$ is the image of the map $\partial D\rightarrow T_x^*D, p\mapsto d_xF_p$.

\begin{lem}
Let $F$ be an enveloping function for $\varphi$ given by $F(p,x)=d(p, x)$. Then every sufficiently small $C^3$ perturbation $\tilde{F}$ of $F$ is an enveloping function of some Finsler metric $\tilde{\varphi}$.
\end{lem}
\begin{proof}
Since $\varphi$ is a Finsler metric,  the image of the map $\partial D\rightarrow T_x^*D, p\mapsto d_xF_p$ is quadratically convex. And $p\mapsto d_x\tilde{F}_p$ is a $C^2$ small perturbation of this map, hence also has quadratically convex image, therefore the image is the unit sphere of some Minkowski norm on $T_x^*D$. And the dual norm $\tilde{\varphi}$ is a Finsler norm at $x$.
\end{proof}





\begin{definition}
([2]) Let $f$ be a function on a Finsler manifold $(M, \varphi)$ that supports $M$, and $\gamma:[a,b]\rightarrow M$ be a geodesic. We say that $\gamma$ is \textbf{calibrated by $f$} if $f(\gamma(t_2))-f(\gamma(t_1))=t_2-t_1,$ for any $ a\leq t_1<t_2\leq b$.
\end{definition}

The \textbf{(Finslerian) gradient} of a distance like function $f:D\rightarrow \mathbb{R}$ at $x\in D$, denoted $\textup{grad}f(x)$, is defined to be the unit tangent vector $v\in U_x D$ such that $d_xf(v)=1$. If $\gamma$ is calibrated by $f$, then for all points on $\gamma$, the tangent vector of $\gamma$ coincide with the gradient of $f$.

\begin{lem}
If we have an enveloping function $F$ on a Finsler disc $(D, \varphi)$, then $D$ has no conjugate points.
\end{lem}
\begin{proof}

If $f$ is a distance-like function on $D$, then any integral curve of $\textup{grad}f$ is a minimal geodesic. In fact, let $\gamma: [a, b]\rightarrow D$ be such a unit-speed curve and $a\leq t_1<t_2\leq b$. Then for all $ s\in(a,b), df_{\gamma(s)}(\gamma'(s))=1$ since $\gamma'(s)$ is the gradient of $f$ at $\gamma(s)$. Thus
\begin{displaymath}
t_2-t_1\geq d(\gamma(t_1),\gamma(t_2))\geq f(\gamma(t_2))-f(\gamma(t_1))=\int_{t_1}^{t_2}df_{\gamma(s)}(\gamma'(s))ds= t_2-t_1. 
\end{displaymath}
This implies $d(\gamma(t_1),\gamma(t_2))=t_2-t_1=f(\gamma(t_2))-f(\gamma(t_1))$. Hence $\gamma$ is a minimal geodesic and it is calibrated by $f$.

Now let $\sigma: [a,b]\rightarrow D$ be a geodesic. Since a geodesic is a local minimizer, we can find $\delta>0$ such that $d(\sigma(a), \sigma(a+\delta))=\delta$. Let $p\in \partial D$ be a point such that $d_{\sigma(a)}F_p$ is the dual to $\sigma'(a)$, then the integral curve $\gamma$ of $\textup{grad}F_p$ with $\gamma(a)=\sigma(a)$ is a minimal geodesic calibrated by $F_p$. Since $\gamma$ and $\sigma$ are geodesics with the same starting point and initial direction, $L(\gamma)=b-a=L(\sigma)$, therefore $\gamma=\sigma$. Therefore $\sigma$ is a minimal geodesic. This implies any geodesic is a minimal one, so $M$ has no conjugate points.


\end{proof}

\begin{rem}
By the proof of Lemma 2, if a geodesic $\gamma$ is calibrated by a distance-like function, then $\gamma$ is minimal.
\end{rem}

\section{Proof of the main theorem}

Lei $\psi: U_0\rightarrow \mathbb{R}^2$ be a local chart around $p_0$ mapping $p_0$ to origin. Since $\psi$ is a diffeomorphism we can define the metric on $\psi(U_0)$ simply by pushing forward that on $U_0$ through $\psi$. Once we get this isometric embedding, we can assume the image of $\psi$ is $U_{\epsilon}:=\{(x, y)\in\mathbb{R}^2|x^2+y^2<\epsilon^2\}$, and we identify $U_0$ with $U_{\epsilon}$. By choosing small $\epsilon$  and let $D_{\epsilon}$ be the closure of $U_{\epsilon}$, we get a simple Finsler metric on $D_{\epsilon}$. Let $\varphi_0$ be a constant Finsler metric on $\mathbb{R}^2$ which is identical to $\varphi|_{T_{p_0} D}$. For $p\in\partial D_{\epsilon}$, let $\gamma^0_p:[-a_0, b_0]\rightarrow D_{\epsilon}$ be the geodesic in the Finsler disk $(D_{\epsilon}, \varphi_0)$ with $\gamma_p(0)=p_0$ and $\gamma_p(b_0)=p, \gamma_p(a_0)=-p$. Let $\gamma_p: [-a, b]\rightarrow D_{\epsilon}$ be the geodesic in $(D_{\epsilon}, \varphi)$ with endpoints on $\partial D_{\epsilon}$ and $\gamma_p(0)=p_0, \gamma'_p(0)=(\gamma^{0}_p)'(0)$. Then we can define a function $F$ on $\partial D_{\epsilon}\times D_{\epsilon} $ by the following: if $x$ lies on the left hand side of the direction of $\gamma_p$, then $F(p,x):=d(x, \gamma_p)$, otherwise define $F(p, x):=-d(\gamma_p, x)$. Then $F$ is a $C^k$ enveloping function for $\varphi$.

 Similar as above we get a $C^{\infty}$ enveloping function $F^0$ on $\partial D_{\epsilon}\times \mathbb{R}^2$ for the constant metric $\varphi_0$. By choosing $\epsilon$ small we may assume $F$ is a $C^k$ small perturbation of $F^0|_{\partial D_{\epsilon}\times D_{\epsilon}}$. Extend $F$ to $\partial D_{\epsilon}\times \mathbb{R}^2$  so that $F$ is a $C^k$ small perturbation of $F^0$. 

Take a large $r$ and let $g$ be a function on $\mathbb{R}^2$ with value 1 on $D_{\epsilon}$ and value 0  outside $D_{r}$. By choosing $r$ large enough we may assume $g$ has very small $i$th($1\leq i\leq k$) derivatives, and define a function $\tilde{F}$ on $\partial D_{\epsilon}\times \mathbb{R}^2$ by 
$$\tilde{F}(p,x)=F(p,x)g(x)+F_0(p,x)(1-g(x)).$$
Then $\tilde{F}$ is a $C^k$ small perturbation of $F_0$. By Lemma 1 it is an enveloping function for some Finsler metric $\tilde{\varphi}$ on $TD_r$. $\tilde{\varphi}$ agrees with $\varphi$ on $TD_{\epsilon}$, and it agrees with $\varphi_0$ on $T(\mathbb{R}^2\backslash D_{r})$. Therefore we can extend $\tilde{\varphi}$ outside $TD_r$ by $\varphi_0$ and get a $C^k$ smooth Finsler metric on $\mathbb{R}^2$. Take an integer $l>r$, then $\tilde{\varphi}$ is also a metric on $\mathbb{T}^2:=\mathbb{R}^2/(2l\mathbb{Z})^2$. By Lemma 2 we know that $\tilde{\varphi}$ has no conjugate points.\\

Suppose $\varphi$ is symmetric, use the same notations as above, then $\gamma_p$ and $\gamma_{-p}$ are the same curve with different directions. Therefore we have 
\begin{equation}
F(p,x)=-F(-p, x), \tag{*}
\end{equation}
 for all $x\in D_{\epsilon}$. Extend $F$ to $\partial D_{\epsilon}\times \mathbb{R}^2$ so that (*) holds. Repeating the procedures as above we get a function $\tilde{F}$ on $\partial D_{\epsilon}\times \mathbb{R}^2$. Now define 

$$\tilde{d}(x,y)=\max_{p\in\partial D_{\epsilon}}\tilde{F}(p,x)-\tilde{F}(p,y),$$
then $\tilde{d}$ is symmetric and it is $C^k$ close to $d_0$, which is the metric on $\mathbb{R}^2$ generated by $\varphi_0$. As we get such metric $\tilde{d}$, we can define a Finsler metric on the tangent bundle in the following way: for $x\in \mathbb{R}^2, v\in T_x \mathbb{R}^2$, let $c:(-\epsilon_1, \epsilon_1)\rightarrow D_r$ be a curve with $c(0)=x, c'(0)=v$. Define $$\tilde{\varphi}(x,v):=\lim_{t\rightarrow 0}\frac{\tilde{d}(x,c(t))}{t}.$$
Then $\tilde{F}$ is an enveloping function for $\tilde{\varphi}$. By symmetry of $\tilde{d}$ we get symmetry of $\tilde{\varphi}$.\\

\section{Acknowledgement}
The author would like to thank Dimitri Burago for suggesting this problem, improvement of content and numerous discussion with him.

\end{document}